\documentclass[11pt,a4paper]{article}

\newcommand{\G}{\mathcal{G}}
\newcommand{\F}{\mathcal{F}}
\renewcommand{\P}{\mathcal{P}}
\newcommand{\M}{\mathcal{M}}

\usepackage[utf8]{inputenc}
\usepackage[T1]{fontenc}
\usepackage[british]{babel}
\usepackage{lmodern}
\usepackage{amsmath,amssymb,amsthm,mathtools}
\usepackage{enumitem}
\usepackage{microtype}
\usepackage[margin=30mm]{geometry}
\usepackage{hyperref}

\newtheorem{theorem}{Theorem}[section]
\newtheorem{lemma}[theorem]{Lemma}
\newtheorem{proposition}[theorem]{Proposition}
\newtheorem{corollary}[theorem]{Corollary}

\theoremstyle{definition}
\newtheorem{definition}[theorem]{Definition}
\newtheorem{remark}[theorem]{Remark}
\newtheorem{assumption}[theorem]{Assumption}

\title{Definability from Factorised Symmetry in Ultrapowers\\[2mm]
\large (Affine Geometries and Beyond)}
\author{Mike Stannett\\[2mm]
\small Department of Computer Science\\University of Manchester, UK}
\date{28 July 2026}

\begin{document}
\maketitle

\begin{abstract}
A recent theorem of Madar\'asz
characterises the parameter-free concepts of a finitely field-definable coordinate
geometry by invariance under its affine automorphisms. This paper isolates the
affine-geometric ingredient in that proof and extends the argument to include additional finite-relational
coordinate geometries. The new hypothesis is \emph{semilinear faithfulness}: in every
ultrapower, each automorphism of the induced geometry factors as an affine automorphism
of the geometry followed by the componentwise action of an automorphism of the expanded
base structure. Under this hypothesis, every parameter-free ambient-definable relation is
a concept of the geometry exactly when it is preserved by the affine automorphism group.
A corresponding dual inclusion theorem is obtained for concept sets. The framework
recovers the original finitely field-definable theorem and applies beyond pure-field
definability, including a named-scalar geometry and a frame-expanded geometry over an
exponential field. A final section records the analogous factorisation principle for
uniformly coded symmetry groups in arbitrary, including non-geometric, one-sorted structures.
\end{abstract}



\section{Introduction and overview}
\label{sec:1}

Felix Klein's Erlangen programme studies the properties left invariant by a chosen
transformation group~\cite[\S\S1--2]{Klein1872}.
The starting point for this paper is a theorem of \cite{MSS26}
that gives an Erlangen-style
connection between \emph{definability} in a coordinate geometry and
\emph{invariance} under its affine automorphisms. Our aim is to isolate the precise affine-geometric
input used by its proof and to identify a hypothesis under which the
same proof mechanism survives when the coordinate field is expanded by
additional structure.

The proof proceeds by transferring an invariance
statement to ultrapowers, using an affine-geometric decomposition of
ultrapower automorphisms into an affine part and a coordinatewise
base-structure automorphism, and then returning from ultrapower
invariance to definability via a Svenonius-type criterion (in an
ultrapower-only form due to András Simon; see \cite[Theorem~6.12]{MSS26}).

In the original finitely field-definable (FFD) setting of
\cite{MSS26}, the semilinear decomposition arising from the Fundamental Theorem of Affine Geometry (FTAG)
 yields a residual
base-field automorphism. Because all primitive relations are definable
in the pure field/ordered-field reduct, that residual automorphism
automatically preserves them, and the proof goes through.

In expanded settings this automatic preservation can fail. We therefore
introduce \emph{semilinear faithfulness}, a structural hypothesis
asserting that in every ultrapower, every automorphism of the geometry
factors as an affine automorphism of the geometry composed with the
coordinatewise action of an automorphism of the \emph{expanded} base
structure. Under semilinear faithfulness we recover an Erlangen-style
characterisation theorem (Theorem \ref{thm:6.1}) and the corresponding duality
between definability and symmetry (Theorem \ref{thm:7.2}), and we give examples
showing that the hypothesis genuinely extends beyond the pure-field
case.

\section{Framework: coordinate geometries, concepts, and affine automorphisms}
\label{sec:2}

\subsection{Ambient base structures and coordinate space}
\label{sec:2.1}

Fix an integer dimension $d\ge 2$, and let $\F$ be a first-order structure expanding either:

\begin{itemize}
\item
  an \emph{ordered field} $\langle F; +,\cdot,0,1,<\rangle$, or
\item
  a \emph{field} $\langle F; +,\cdot,0,1\rangle$.
\end{itemize}

The language of $\F$ may include additional predicate,
function, and constant symbols beyond the field/ordered-field reduct.
When working in the field setting we assume throughout that $\F$
has more than two elements (i)n the ordered-field setting
this is automatically true, because ordered fields are infinite).
We work on the coordinate domain $F^d$. For each $n\ge 1$ we
identify $(F^d)^n\cong F^{dn}$ in the usual way.

Unless explicitly stated otherwise, all definability in this paper is
\emph{parameter-free}.

\subsection{The key ternary relation \texorpdfstring{$\mathsf K$}{K}}
\label{sec:2.2}

To treat ordered-field and field settings uniformly, we work with a
distinguished ternary relation symbol $\mathsf K$ (the ``key relation''),
interpreted as follows.

\begin{itemize}
\item
  \emph{Ordered-field case.} $\mathsf K$ is interpreted as the
  standard \emph{betweenness} relation $\mathsf{Bw}$ on $F^d$: \[
  \mathsf{Bw}(a,b,c) \iff \exists t\,\big(0\le t\le 1\ \wedge\ b=(1-t)a+tc\Big),
  \] where arithmetic is interpreted componentwise in $F^d$.
\item
  \emph{Field case.} $\mathsf K$ is interpreted as the standard
  \emph{collinearity} relation $\mathsf{Col}$ on $F^d$: \[
  \mathsf{Col}(a,b,c) \iff (a=b)\ \vee\ (b=c)\ \vee\ (a=c)\ \vee\ \exists t\,\big(b=(1-t)a+tc\big).
  \]
\end{itemize}

In either setting, the intended interpretation of $\mathsf K$ is
parameter-free definable in the appropriate reduct of $\F$
(ordered-field reduct for $\mathsf{Bw}$, field reduct for
$\mathsf{Col}$), and hence in $\F$.

\begin{remark}[Geometric vs. non-geometric structures]
\label{rem:2.1}

Except in Section \ref{sec:9} where we deliberately consider more general structures,
all examples treated in this paper are ``geometric'' in the sense that
$\mathsf K$ is included among the primitives. 
\end{remark}

\begin{remark}[Tuple metavariables and formulas]
	\label{rem:2.3}
	 In definitions
and proofs we frequently use tuple metavariables such as $x\in F^d$ or
$\bar x\in (F^d)^n$. Formally, when writing formulas in the language
of $\F$ these stand for $d$-tuples (respectively
$dn$-tuples) of object variables ranging over $F$, with the obvious
componentwise interpretation.
\end{remark}

\subsection{Coordinate geometries}
\label{sec:2.3}

A \emph{finite-relational $\F$-definable coordinate geometry
on $F^d$} is a structure 
\[
\G\ :=\ \langle F^d;\ \mathsf K,\ P_1,\dots,P_k\rangle
\] in the finite relational language $L_{\G}:=\{\mathsf K,P_1,\dots,P_k\}$, such that each primitive relation $P_i^{\G}\subseteq (F^d)^{n_i}$ is
parameter-free $\F$-definable.
When we need to emphasise the ambient base structure we write
$\G(\F)$.

\begin{remark}[Relation with the earlier framework]
\label{rem:2.4} 
In
\cite{MSS26}, a coordinate geometry is required to \emph{define} the
relevant key relation $\mathsf K$; it need not name $\mathsf K$ as a
primitive. We work throughout with the corresponding definitional
expansion in which $\mathsf K$ is named. Since $\mathsf K$ is
already parameter-free definable in the original geometry, adjoining it
as a primitive changes neither the parameter-free concepts nor the
automorphism group. Thus results from \cite{MSS26} apply without loss of
generality in our $\mathsf K$-named presentation.
\end{remark}

\subsection{Concepts}
\label{sec:2.4}

Let $\G$ be a coordinate geometry. For each $n\ge 1$, let
$\mathrm{Conc}_n(\G)$ be the set of all relations
$R\subseteq (F^d)^n$ for which there exists a parameter-free
$L_{\G}$-formula $\varphi(\bar x)$ such that for all
$\bar a\in (F^d)^n$, \[
\bar a\in R \iff \G\models \varphi(\bar a).
\] Write \[
\mathrm{Conc}(\G)\ :=\ \bigcup_{n\ge 1}\mathrm{Conc}_n(\G).
\] Elements of $\mathrm{Conc}(\G)$ are called \emph{concepts} of
$\G$.

\begin{remark}[Concepts for arbitrary structures]
	\label{rem:2.2}
	Although we
	introduce ${\rm Conc}_n(\G)$ for coordinate geometries, $\G$, the same
	definition makes sense for any first-order structure $\M$ (in its own
	language). We will use this in Section \ref{sec:9.1}.
\end{remark}

\subsection{Affine transformations and affine automorphisms}

Let $\mathrm{Aff}(F^d)$ denote the group of affine transformations of
$F^d$, i.e.~maps of the form \[
 x\mapsto Ax+b,
\] where $A\in \mathrm{GL}_d(M)$ and $b\in F^d$.

Let $\mathrm{Aut}(\G)$ denote the automorphism group of the first-order
structure $\G$. Define the \emph{affine automorphism group} of $\G$
by \[
\mathrm{AffAut}(\G)\ :=\ \mathrm{Aut}(\G)\cap \mathrm{Aff}(F^d).
\] Thus $\mathrm{AffAut}(\G)$ consists of those automorphisms of the
geometry whose underlying permutation of $F^d$ is affine with respect
to the ambient field operations on $F$.

\subsection{Preservation}
\label{sec:2.6}

If $R\subseteq (F^d)^n$ and $\alpha\in \mathrm{Sym}(F^d)$, define
the image relation \[
\alpha(R)\ :=\ \{(\alpha(a_1),\dots,\alpha(a_n)):(a_1,\dots,a_n)\in R\}.
\] We say that $R$ is \emph{preserved by} a subgroup
$H\le \mathrm{Sym}(F^d)$ if $\alpha(R)=R$ for all $\alpha\in H$.
In particular, we will frequently consider preservation by
$\mathrm{AffAut}(\G)$.

\section{Model-theoretic tools: ultrapowers and definability from invariance}
\label{sec:3}

This section fixes ultrapower notation and records the
definability-from-invariance criterion that drives the return step from
``preserved by automorphisms'' to ``definable''.

\subsection{Ultrapowers and induced coordinate geometries}
\label{sec:3.1}

Let $U$ be an ultrafilter on a set $I$. Write $\mathcal \F^U$ for
the ultrapower of $\F$ by $U$, with underlying set $F^U$.
The coordinate domain of the induced geometry will be $(F^U)^d$.

\begin{definition}[Induced coordinate geometry in an ultrapower]
	\label{def:3.1}
Let $G=\langle F^d;\mathsf K,P_1,\dots,P_k\rangle$ be a
finite-relational $\F$-definable coordinate geometry. Define
\[
\G(\F^U)\ :=\ \langle (F^U)^d;\ \mathsf K^U,\ P_1^U,\dots,P_k^U\rangle,
\] where
\begin{itemize}[leftmargin=2em]
\item $\mathsf K^U$ is interpreted as $\mathsf{Bw}$ or
$\mathsf{Col}$ in $(F^U)^d$, according to whether $\F$
expands an ordered field or a field;
\item each $P_i^U\subseteq ((F^U)^d)^{n_i}$ is interpreted by the same
parameter-free $\F$-formula that defines $P_i$ in
$\F$.
\end{itemize}
\end{definition}


\subsection{Transfer of parameter-free definability}
\label{sec:3.2}

Let $R\subseteq (F^d)^n$ be parameter-free $\F$-definable, defined (say)
by some parameter-free $\F$-formula $\varphi(\bar x)$. 
Define $R^U\subseteq ((F^U)^d)^n$ by the same formula in
$\F^U$: \[
\bar a\in R^U \iff \F^U\models \varphi(\bar a).
\] By Łoś's theorem, this is equivalent to the usual ultrapower
interpretation of $R$.

For an arbitrary relation $S\subseteq (F^d)^n$, not assumed to be
$\F$-definable, we use $S^U$ for its ordinary ultrapower relation on
$((F^d)^U)^n$, transported to $((F^U)^d)^n$ along the canonical
identification $(F^d)^U\cong(F^U)^d$. When $S$ is parameter-free
$\F$-definable, this agrees with the interpretation by the same formula
just described.

In particular, every primitive relation of $\G$ transfers to the
corresponding primitive relation of $\G(\F^U)$, and every
parameter-free $L_{\G}$-definable relation of $\G$ transfers to a
parameter-free $L_{\G}$-definable relation of $\G(\F^U)$.

\subsection{Definability from invariance: Simon's ultrapower criterion}
\label{sec:3.3}

Following \cite{MSS26} we will use a definability-from-invariance criterion of Svenonius type,
in an ultrapower-only form due to Andr\'as Simon.

Fix a first-order language $L$, an $L$-structure $\M$ with universe $M$, and an
$n$-ary relation $R\subseteq M^n$. Write
$\M_R:=\langle A;L\cup\{R\}\rangle$ for the expansion of $\M$ by a new
predicate symbol interpreted as $R$.

Given an ultrafilter $U$ on a set $I$, write $\M^U$ for the
ultrapower. Define the induced relation $R^U\subseteq (M^U)^n$ by \[
([a^1]_U,\dots,[a^n]_U)\in R^U \iff \{i\in I : (a^1(i),\dots,a^n(i))\in R\}\in U.
\]

\begin{theorem}[Simon]
\label{thm:3.2}
The following are equivalent.

\begin{enumerate}
\item
  The relation $R$ is parameter-free definable in $\M$ by an
  $L$-formula.
\item
  For every ultrafilter $U$ and every $\alpha\in \mathrm{Aut}(\M^U)$,
  one has $\alpha(R^U)=R^U$.
\end{enumerate}
\end{theorem}

\begin{proof}
	The implication $(1)\Rightarrow (2)$ is immediate since
automorphisms preserve all parameter-free definable relations, and this
transfers to ultrapowers. For the converse, the criterion is due to
András Simon; a proof is given in \cite[Theorem~6.12]{MSS26}.
\end{proof}

\subsection{Application to concepts}

The following is the form in which we will invoke Simon's criterion
throughout.

\begin{corollary}[Definability from ultrapower invariance]
\label{cor:3.4}
Let
$\G$ be a coordinate geometry on $F^d$ in language $L_{\G}$ and let
$R\subseteq (F^d)^n$ be any $n$-ary relation. Assume that for every
ultrafilter $U$ and every $\alpha\in \mathrm{Aut}(\G(\F^U))$,
\[
\alpha(R^U)=R^U.
\] Then $R\in \mathrm{Conc}_n(\G)$.
\end{corollary}

\begin{proof}
	Apply Theorem \ref{thm:3.2} with $A:=G$. Since
$\G^U\cong \G(\F^U)$, the hypothesis implies that every
automorphism of $\G^U$ preserves $R^U$. Therefore $R$ is
parameter-free definable in $\G$.
\end{proof}

\section{Baseline: Madarász's theorem and the proof mechanism}
\label{sec:4}

This section records the original representation theorem and isolates
the proof ingredients that will be re-used. We emphasise how the
affine-geometric input enters: it provides a semilinear factorisation of
automorphisms in ultrapowers into an affine part and a componentwise
base-field automorphism part. We will state the relevant
affine-geometric theorem precisely later, when it is first invoked
formally.

\subsection{FFD coordinate geometries and Madar\'asz's theorem}
\label{sec:4.1}

The fundamental result we wish to generalise is Theorem 4.2 of \cite{MSS26}, which states:

\begin{theorem}[Madar\'asz]
	\label{thm:mss26}
	Let $\mathcal{F}$ be an ordered field or a field that has more than two elements, let $\mathcal{G}$
	be an FFD coordinate geometry over $\mathcal{F}$, and let $\mathsf{R}$ be a relation of points of $F^d$. Then the
	following statements are equivalent:
	\begin{itemize}
		\item[(i)] $\mathsf{R}$ is a concept of $\mathcal{G}$ (i.e. $\mathsf{R}$ is definable in $\mathcal{G}$).
		\item[(ii)] $\mathsf{R}$ is definable over $\mathcal{F}$ and is closed under automorphisms of $\mathcal{G}$.
		\item[(iii)] $\mathsf{R}$ is definable over $\mathcal{F}$ and is closed under affine automorphisms of $\mathcal{G}$.
	\end{itemize}
\end{theorem}
\begin{proof}This is \cite[Theorem~4.2]{MSS26}.\end{proof}
In this theorem, $F$ is the universe of $\mathcal{F}$, and the dimension $d$ is a fixed integer satisfying $d \ge 2$. Saying that $\mathcal{G}$ is \emph{field-definable} means that all the relations of $\mathcal{G}$ are definable over $\mathcal{F}$; it is \emph{finitely field definable } (FFD) iff it is field-definable and contains only finitely many relations.

The two parts of the FFD hypothesis play different roles in the proof.
Field-definability ensures that the residual field-automorphism
component preserves every primitive relation, and indeed every relation
definable in the pure field or ordered-field reduct. Finiteness of the
relational language is used separately: it allows the assertion, ``an
affine map preserves all primitives,'' to be written as a single
first-order sentence and hence transferred to ultrapowers.

In this paper we are specifically interested in affine automorphisms, so the form in which
we will use Theorem~\ref{thm:mss26} is:

\begin{corollary}
\label{thm:4.1}
Assume $\F$ expands an ordered field, or a field
with more than two elements. Let $\G$ be an FFD coordinate geometry on $F^d$ with
$d\ge 2$. Let $R\subseteq (F^d)^n$ be field-definable
(parameter-free in the pure field/ordered-field language). Then \[
R\in \mathrm{Conc}(\G)\ \Longleftrightarrow\ \big(\text{$R$ is preserved by }\mathrm{AffAut}(\G)\Big).
\]
\end{corollary}
\begin{proof}
	This is implication (iii)$\Rightarrow$(i) in Theorem
\ref{thm:mss26}, together with the easy direction
(i)$\Rightarrow$(iii).
\end{proof}

\label{sec:4.2}

The same framework yields a useful comparison theorem.

\begin{theorem}[{\cite[Theorem~4.5(ii)]{MSS26}}]
\label{thm:4.2}
Assume that $\mathcal{F}$ is an ordered field or a field with
more than two elements, and that $\mathcal{G}$ and $\mathcal{G}'$ are
FFD coordinate geometries over $\mathcal{F}$. Then 
\[
\mathrm{Conc}(\mathcal{G}) \subseteq \mathrm{Conc}(\mathcal{G}')
    \Longleftrightarrow
\mathrm{AffAut}(\mathcal{G}') \subseteq \mathrm{AffAut}(\mathcal{G}).
\]
$\hfill \Box$
\end{theorem}

\subsection{The proof mechanism}
\label{sec:4.2}

The proof operates by combining various observations, as expressed in the following lemmas.

\begin{lemma}[First-order expressibility and transfer of affine
preservation]
\label{lem:4.3}
Let $\G$ be a finite-relational coordinate geometry and
let $R\subseteq (F^d)^n$ be a parameter-free $\F$-definable
relation. The statement
\begin{center}
\emph{every affine automorphism of $\G$ preserves $R$}
\end{center}
is expressible by a first-order sentence in the language of
$\F$ (uniformly in formulas defining the primitives of $\G$
and defining $R$), and hence transfers to ultrapowers.
\end{lemma}
\begin{proof}
	Choose parameter-free $\F$-formulas
$\kappa,\pi_1,\dots,\pi_k$ defining $\mathsf K,P_1,\dots,P_k$, and a
parameter-free formula $\varphi$ defining $R$. For a matrix tuple
$B\in M^{d\times d}$ and a vector $b\in F^d$, write
$T_{B,b}(x)=Bx+b$. The required assertion is expressed by the sentence
\[
\begin{aligned}
\forall B\,\forall b\,\big(&\det(B)\neq 0 \\
&{}\wedge \forall x_1x_2x_3\,
 \big(\kappa(x_1,x_2,x_3)\leftrightarrow \\
&\hspace{43mm}
 \kappa(T_{B,b}x_1,T_{B,b}x_2,T_{B,b}x_3)\Big) \\
&{}\wedge \bigwedge_{i=1}^k \forall \bar x\,
 \big(\pi_i(\bar x)\leftrightarrow\pi_i(T_{B,b}\bar x)\Big) \\
&{}\Longrightarrow \forall \bar x\,
 \big(\varphi(\bar x)\leftrightarrow
 \varphi(T_{B,b}\bar x)\Big)\Bigg).
\end{aligned}
\] Here $T_{B,b}\bar x$ means that $T_{B,b}$ is applied to each
point in the tuple. The finite conjunction over the primitives is
available because $\G$ is finite-relational. The determinant is a fixed
polynomial in the matrix entries and is therefore definable in the field
language. The displayed biconditionals say exactly that the coded affine
bijection is an automorphism of $\G$ and that it preserves $R$.
Transfer to ultrapowers is now an immediate application of Łoś's
theorem.
\end{proof}

\begin{proposition}[Semilinear decomposition for coordinate geometries]
	\label{prop:4.4}
Assume $F$ is an ordered
field, or a field with more than two elements, and let $d\ge 2$. Let \[
\G\ :=\ \langle F^d;\ \mathsf K,\ P_1,\dots,P_k\rangle
\] be a finite-relational coordinate geometry whose primitive relations
are parameter-free definable in the pure field/ordered-field language of
$F$. Equivalently, after forgetting that $\mathsf K$ has been named
as a primitive, this is an FFD coordinate geometry in the
sense of \cite{MSS26}. Then every automorphism
$\alpha\in\mathrm{Aut}(\G)$ admits a (unique) factorisation \[
\alpha\ =\ A\circ \widetilde\sigma,
\] where $A\in\mathrm{AffAut}(\G)$ and $\sigma\in\mathrm{Aut}(\F)$
(automorphisms taken in the pure field language, or in the pure
ordered-field language in the ordered case), and where
$\widetilde\sigma:F^d\to F^d$ acts componentwise.
\end{proposition}

\begin{proof}
This is the finite-relational, $\mathsf K$-named special
case of \cite[Proposition~6.1]{MSS26}. The proposition in \cite{MSS26}
does not require finiteness and derives the factorisation, including
uniqueness, from the Fundamental Theorem of Affine Geometry. Remark \ref{rem:2.4}
explains why naming $\mathsf K$ does not change the automorphism
group.
\end{proof}

\begin{lemma}[Affine-geometric semilinear factorisation in
	FFD ultrapowers]
	\label{lem:4.5}
Assume $\F$ expands an
ordered field, or a field with more than two elements, and let $d\ge 2$. Let $\G$
be an FFD coordinate geometry on $F^d$ and let $U$ be an
ultrafilter. Then every 
\[
\alpha\in \mathrm{Aut}\big(\G(\F^U)\big)
\] admits a factorisation \[
\alpha\ =\ A\circ \widetilde\sigma
\] with $A\in\mathrm{AffAut}\big(\G(\F^U)\big)$ and
$\sigma\in\mathrm{Aut}(\F^U)$, where $F^U$ denotes the pure
field/ordered-field reduct of $\F^U$.
\end{lemma}

\begin{proof}
	Because $\G$ is FFD, each primitive relation of
$\G(\F^U)$ is parameter-free definable in the pure
field/ordered-field reduct $\F^U$ (by transfer of parameter-free
definability). Thus $\G(\F^U)$ is a field-definable coordinate
geometry over $F^U$ in the sense of \cite{MSS26}. In the
unordered field case, $|F|>2$ implies $|F^U|>2$ by Łoś's theorem, so
Proposition \ref{prop:4.4} applies to $F^U$. In the ordered-field case,
automorphisms of $\F^U$ are taken in the pure ordered-field language
(so they preserve $<$). Applying Proposition \ref{prop:4.4} to
$\G(\F^U)$ yields the required factorisation, with
$A\in\mathrm{AffAut}\big(\G(\F^U)\big)$.
\end{proof}

\begin{lemma}[Residual preservation in the FFD setting]
	\label{lem:4.6}
Assume
$\F$ expands an ordered field, or a field with more than two elements, and
let $d\ge 2$. Let $\G$ be an FFD coordinate geometry on $F^d$, let
$U$ be an ultrafilter, and let $\sigma\in\mathrm{Aut}(\F^U)$. Then
the componentwise map $\widetilde\sigma:(F^U)^d\to(F^U)^d$ preserves
every relation that is parameter-free definable in $F^U$ (in
particular, each primitive relation of $\G(\F^U)$ and every
field-definable relation on $(F^U)^d$).
\end{lemma}

\begin{proof}
	The map $\sigma$ is an automorphism of $F^U$, so it
preserves all parameter-free definable relations on $\F^U$ in the pure
field/ordered-field language. Applying this coordinatewise yields the
stated preservation on $(F^U)^d$.
\end{proof}

\section{Expansions and semilinear faithfulness}
\label{sec:5}

\subsection{Why expanded-base definability alone is not enough}
\label{sec:5.1}

This section isolates the point at which the baseline proof mechanism
fails if the coordinate field is expanded by additional structure.

Let $\F$ be an expansion of an ordered field (or field), and
let \[
\G=\langle F^d;\ \mathsf K,\ P_1,\dots,P_k\rangle
\] be a finite-relational coordinate geometry whose primitive relations
are parameter-free $\F$-definable.

Suppose $R\subseteq (F^d)^n$ is parameter-free
$\F$-definable and is preserved by $\mathrm{AffAut}(\G)$. The
proof strategy of Section \ref{sec:4} begins by transferring affine preservation to
ultrapowers, and then tries to show that an arbitrary ultrapower
automorphism preserves $R^U$ by factoring it into an affine part and a
``field automorphism part''.

The obstruction is that, in expansions, the residual field automorphism
produced by the affine-geometric factorisation need not be an
automorphism of the expanded structure $\F^U$. In particular,
it may fail to preserve additional primitives of $\F$ (named
constants, predicates, functions), and therefore it may fail to preserve
$\F$-definable relations such as $R$.

There is a second warning sign. If $\mathrm{AffAut}(\G)$ is trivial,
then every relation is invariant under it, so the affine-invariance
condition imposes no restriction. The desired equivalence can then hold
only if every parameter-free relation definable in the ambient expansion
is already a concept of $\G$. This can occur, as the exponential
example in Section \ref{sec:8.2} shows, but it requires the geometry to recover enough of
the expanded base structure; triviality of the affine group by itself
provides no such conclusion.

Accordingly, if the mechanism of \cite{MSS26} is to survive in expansions, one
must add a hypothesis ensuring that the residual semilinear component is
controlled at the level of the expanded base structure, in every
ultrapower.

\label{sec:5.2}


\begin{definition}[Semilinear faithfulness over $\F$]
	\label{def:5.1}
Let $\F$ expand an ordered field (or a field), let $d\ge 2$,
and let $\G=\langle F^d;\ \mathsf K,\ P_1,\dots,P_k\rangle$ be a
finite-relational, parameter-free $\F$-definable coordinate
geometry. We say that $\G$ is \emph{semilinearly faithful over
$\F$} if for every ultrafilter $U$ and every automorphism \[
\alpha\in \mathrm{Aut}\big(\G(\F^U)\big),
\] there exist \[
A\in \mathrm{AffAut}\big(\G(\F^U)\big)\quad\text{and}\quad \sigma\in \mathrm{Aut}(\F^U)
\] such that \[
\alpha = A\circ \widetilde\sigma,
\] where $\widetilde\sigma:(F^U)^d\to (F^U)^d$ acts componentwise: \[
\widetilde\sigma\big((x_1,\dots,x_d)\big) := (\sigma(x_1),\dots,\sigma(x_d)).
\]
\end{definition}


\section{Main theorem: Semilinear Erlangen characterisation}
\label{sec:6}

\begin{theorem}[Semilinear Erlangen characterisation]
	\label{thm:6.1}
Let
$\F$ expand an ordered field (or a field with more than two elements), let
$d\ge 2$, and let $G=\langle F^d;\ \mathsf K,\ P_1,\dots,P_k\rangle$
be a finite-relational, parameter-free $\F$-definable
coordinate geometry. Assume that $\G$ is semilinearly faithful over
$\F$. Let $R\subseteq (F^d)^n$ be any parameter-free
$\F$-definable relation. Then \[
R\in \mathrm{Conc}(\G)\ \Longleftrightarrow\ \big(\text{$R$ is preserved by }\mathrm{AffAut}(\G)\Big).
\]
\end{theorem}

\begin{proof}
($\Rightarrow$) If $R\in \mathrm{Conc}(\G)$ then $R$
is preserved by every automorphism of $\G$, hence by every affine
automorphism.

($\Leftarrow$) Assume $R$ is preserved by $\mathrm{AffAut}(\G)$. By
Lemma \ref{lem:4.3}, the statement that every affine automorphism preserves $R$
transfers to every ultrapower: for every ultrafilter $U$, the relation
$R^U$ is preserved by $\mathrm{AffAut}(\G(\F^U))$.

Fix an ultrafilter $U$ and let
$\alpha\in \mathrm{Aut}(\G(\F^U))$ be arbitrary. By semilinear
faithfulness, write $\alpha=A\circ \widetilde\sigma$ with
$A\in \mathrm{AffAut}(\G(\F^U))$ and
$\sigma\in \mathrm{Aut}(\F^U)$. Then $A$ preserves $R^U$
by the transferred affine-invariance hypothesis. Also
$\widetilde\sigma$ preserves $R^U$ because $R^U$ is parameter-free
definable in $\F^U$ and
$\sigma\in \mathrm{Aut}(\F^U)$. Hence $\alpha(R^U)=R^U$.

Since $U$ and $\alpha$ were arbitrary, every automorphism of every
ultrapower preserves $R^U$. By Corollary \ref{cor:3.4},
$R\in \mathrm{Conc}(\G)$.
\end{proof}

\begin{remark}[Role of the field and dimension assumptions]
	\label{rem:6.2}
Once
semilinear faithfulness is assumed, the proof of Theorem \ref{thm:6.1} itself does
not use the Fundamental Theorem of Affine Geometry, the restriction
$|F|>2$, or the inequality $d\ge 2$. Those assumptions belong to the
uniform coordinate-geometric framework and are needed when semilinear
faithfulness is established from affine geometry, as in the baseline
theorem and the examples below. The abstract implication of Theorem \ref{thm:6.1}
remains valid whenever Definition \ref{def:5.1} and Lemma \ref{lem:4.3} make sense.
\end{remark}

\label{sec:7}

The Erlangen-style contravariance between definability and symmetry
persists under the generalized hypothesis.

\begin{lemma}[Concepts are ambient definable]
	\label{lem:7.1}
Let $\F$
expand a field or ordered field and let
$G=\langle F^d;\ \mathsf K,\ P_1,\dots,P_k\rangle$ be a coordinate
geometry whose primitives are parameter-free $\F$-definable.
Then every concept of $\G$ is parameter-free $\F$-definable.
\end{lemma}

\begin{proof}
	Each primitive relation of $\G$ is defined by a
parameter-free $\F$-formula. Replacing each occurrence of
$\mathsf K,P_1,\dots,P_k$ in an $L_{\G}$-formula defining a concept by
the corresponding $\F$-formulas yields a parameter-free
$\F$-definition of the same relation.
\end{proof}

\begin{theorem}[Dual inclusion theorem]
	\label{thm:7.2}
Let $\F$ expand
an ordered field (or a field with more than two elements), let $d\ge 2$, and let
$\G, \G'$ be finite-relational, parameter-free
$\F$-definable coordinate geometries on the same domain
$F^d$. Assume that $\G'$ is semilinearly faithful over
$\F$. Then \[
\mathrm{Conc}(\G)\subseteq \mathrm{Conc}(\G')\ \Longleftrightarrow\ \mathrm{AffAut}(\G')\subseteq \mathrm{AffAut}(\G).
\]
\end{theorem}

\begin{proof}
($\Rightarrow$) Suppose
$\mathrm{Conc}(\G)\subseteq \mathrm{Conc}(\G')$ and let
$\alpha\in \mathrm{AffAut}(\G')$. Then $\alpha$ preserves every
relation in $\mathrm{Conc}(\G')$, hence every relation in
$\mathrm{Conc}(\G)$. In particular, $\alpha$ preserves every
primitive relation of $\G$, because primitives are definable (hence
concepts) of $\G$. Therefore $\alpha\in \mathrm{Aut}(\G)$. Since
$\alpha$ is affine, $\alpha\in \mathrm{AffAut}(\G)$.

($\Leftarrow$) Suppose
$\mathrm{AffAut}(\G')\subseteq \mathrm{AffAut}(\G)$. Let
$R\in \mathrm{Conc}(\G)$. By Lemma \ref{lem:7.1}, $R$ is parameter-free
$\F$-definable. Moreover, $R$ is preserved by
$\mathrm{AffAut}(\G)$, hence by $\mathrm{AffAut}(\G')$. Since
$\G'$ is semilinearly faithful, Theorem \ref{thm:6.1} implies
$R\in \mathrm{Conc}(\G')$. Thus
$\mathrm{Conc}(\G)\subseteq \mathrm{Conc}(\G')$.
\end{proof}

\section{Examples beyond pure-field definability}
\label{sec:8}

The following examples confirm that semilinear faithfulness can hold even when the primitives are not definable in the pure field/ordered-field language (and hence, when the geometry is not FFD).

\subsection{Named-scalar geometry}
\label{sec:8.1}

Fix an ordered field expansion \[
\F_c:=\langle F;+ ,\cdot,0,1,<,c\rangle
\] where $c\in F$ is a distinguished constant symbol (which we assume is not
parameter-free definable in the pure ordered-field reduct).
Assume $d\ge 2$, and define a ternary relation $Q_c\subseteq (F^d)^3$ by \[
Q_c(x,y,z)\ \iff\ y=x+c(z-x).
\] Let \[
\G_c:=\langle F^d;\ \mathsf{Bw},\ Q_c\rangle,
\] and note that $\G_c$ is a finite-relational, parameter-free
$\F_c$-definable coordinate geometry.

\begin{proposition}
	\label{prop:8.1}
The geometry $\G_c$ is semilinearly faithful
over $\F_c$.
\end{proposition}

\begin{proof}
Fix an ultrafilter $U$ and let
$\alpha\in \mathrm{Aut}(\G_c(\F_c^U))$. Then $\alpha$
preserves $\mathsf{Bw}$, so $\alpha$ is an automorphism of the
ordered affine geometry $\langle (F^U)^d;\mathsf{Bw}\rangle$ over the
ordered-field reduct $F^U$ of $\F_c^U$. By Proposition \ref{prop:4.4}
applied to $\langle (F^U)^d;\mathsf{Bw}\rangle$, we may write \[
\alpha \ =\ A\circ \widetilde\sigma,
\] where $A\in \mathrm{Aff}((F^U)^d)$ is affine and $\sigma$ is an
automorphism of the ordered-field reduct of $\F_c^U$.

We first note that every affine map preserves the relation $Q_c$.
Indeed, if $A(x)=Lx+b$ with $L\in \mathrm{GL}_d(F^U)$, then \[
A\big(x+c(z-x)\big) \ =\ Lx+b+c(Lz-Lx)\ =\ A(x)+c(A(z)-A(x)),
\] so $Q_c(x,y,z)\Rightarrow Q_c(Ax,Ay,Az)$. Applying the same
calculation to the affine inverse $A^{-1}$ gives the converse
implication. Thus every affine map respects $Q_c$, in the sense that
it carries the relation onto itself.

Since $\alpha$ is an automorphism of $\G_c(\F_c^U)$ it
preserves $Q_c$, and we have just seen that $A$ preserves $Q_c$.
Hence $\widetilde\sigma = A^{-1}\circ \alpha$ also preserves $Q_c$.
In particular, \[
Q_c(0,\ c e_1,\ e_1)
\] holds in $(F^U)^d$ (where $e_1=(1,0,\dots,0)$), so applying
$\widetilde\sigma$ yields \[
Q_c\big(0,\ \sigma(c)e_1,\ e_1\big).
\] Unwinding the definition of $Q_c$ with $x=0$ and $z=e_1$, this
implies $\sigma(c)e_1 = c e_1$, hence $\sigma(c)=c$. Therefore
$\sigma$ is an automorphism of the expanded structure
$\F_c^U=\langle F^U;+,\cdot,0,1,<,c\rangle$. Moreover,
$\widetilde\sigma$ preserves $\mathsf{Bw}$, because $\sigma$ is an
ordered-field automorphism, and it preserves $Q_c$ by the preceding argument.
Thus $\widetilde\sigma\in\mathrm{Aut}(\G_c(\F_c^U))$.

Finally, define $A':=\alpha\circ (\widetilde\sigma)^{-1}$. Then $A'$
is affine, and since both $\alpha$ and $\widetilde\sigma$ are
automorphisms of $\G_c(\F_c^U)$, the map $A'$ is an affine
automorphism of $\G_c(\F_c^U)$. Thus
$\alpha=A'\circ \widetilde\sigma$ with
$A'\in \mathrm{AffAut}(\G_c(\F_c^U))$ and
$\sigma\in \mathrm{Aut}(\F_c^U)$, as required.
\end{proof}

\begin{remark}
	\label{rem:8.2}
Given our assumption on $c$, the
primitive $Q_c$ is not parameter-free definable in the pure
ordered-field language. For if a pure ordered-field formula defined
$Q_c$, then $c$ itself would be defined as the unique scalar $t$
satisfying \[
Q_c(0,t e_1,e_1).
\] The points $0$ and $e_1$ and the map $t\mapsto t e_1$ are
definable in the pure ordered-field language, while the displayed
condition is equivalent to $t=c$. This contradicts the assumption that
$c$ is not parameter-free definable in the pure ordered-field reduct.
Hence $\G_c$ is not FFD in the sense of \cite{MSS26}.
\end{remark}

\subsection{Frame-expanded exponential geometry}
\label{sec:8.2}

Retain the standing assumption $d\ge 2$. Let $\F_{\exp}$ be an expansion of an ordered field by a unary
function symbol $\exp$:
\[
\F_{\exp}\ :=\ \langle F; +,\cdot,0,1,<,\exp\rangle .
\]
Although the ambient expansion has a function symbol, the coordinate geometry
below is finite-relational: it records $\exp$ only through a binary relation
naming the relevant graph. (For the motivating case, take $\F=\mathbb R$ with its usual
exponential function.)

Define unary relations $O,E_1,\dots,E_d\subseteq F^d$ and binary relation $\Gamma_{\exp}\subseteq (F^d)^2$  by
\begin{itemize}[leftmargin=2em]
\item $O(x)$ if and only if $x=(0,\dots,0)$;
\item $E_i(x)$ if and only if $x=e_i$, the $i$th standard basis point.
\item $\Gamma_{\exp}(p,q)\ \iff\ \big(p_2=\cdots=p_d=0\big)\ \wedge\ \big(q_1=q_3=\cdots=q_d=0\big)\ \wedge\ \big(q_2=\exp(p_1)\big)$, where $p=(p_1,\dots,p_d)$ and $q=(q_1,\dots,q_d)$.
\end{itemize}

\begin{remark}
	\label{rem:8.3}
When
$\F=\mathbb R$, the relation $\Gamma_{\exp}$ is not definable in the
pure ordered-field reduct $\langle \mathbb R;+ ,\cdot,0,1,<\rangle$.
This is because every first-order definable subset of a real closed field is
semialgebraic by quantifier elimination; see
\cite[Theorem~2.6]{Coste2002}. If $\Gamma_{\exp}$ were definable, its
image under the coordinate projection
\[
 (p,q)\longmapsto (p_1,q_2)
\]
would therefore be semialgebraic by
\cite[Corollary~2.4]{Coste2002}. That image is exactly the graph of the
real exponential function. It would follow that $\exp:\mathbb R\to
\mathbb R$ is a semialgebraic function. But every semialgebraic function
$f:(A,+\infty)\to\mathbb R$ is eventually bounded in absolute value by
some power $x^n$, by \cite[Proposition~2.11]{Coste2002}; this contradicts
$\exp(x)/x^n\to\infty$ for every $n$.
\end{remark}

Let \[
\G_{\exp}\ :=\ \langle F^d;\ \mathsf{Bw},\ O,E_1,\dots,E_d,\ \Gamma_{\exp}\rangle .
\] Then $\G_{\exp}$ is a finite-relational, parameter-free
$\F_{\exp}$-definable coordinate geometry.

\begin{lemma}
	\label{lem:8.4}
	\label{lem:8.5}
For every ultrapower $U$,
$\mathrm{AffAut}(G_{\exp}(\F_{\exp}^U)) =
 \mathrm{AffAut}(\G_{\exp}) =
 \{\mathrm{id}\}$.
\end{lemma}
\begin{proof}
	The relations $\{ O, E_1, \dots E_d \}$ specify the standard coordinate frame for $F^d$ (resp.  $(F^U)^d$), so any
	affine map preserving all of these relations must be the identity.
\end{proof}

\begin{proposition}
	\label{prop:8.6}
The geometry $\G_{\exp}$ is semilinearly
faithful over $\F_{\exp}$.
\end{proposition}

\begin{proof}
Fix an ultrafilter $U$ and let
$\alpha\in \mathrm{Aut}(\G_{\exp}(\F_{\exp}^U))$. Then
$\alpha$ preserves $\mathsf{Bw}$, so $\alpha$ is an automorphism
of the ordered affine geometry $\langle (F^U)^d;\mathsf{Bw}\rangle$
over the ordered-field reduct $\F^U$ of $\F_{\exp}^U$. By
Proposition \ref{prop:4.4} applied to $\langle (F^U)^d;\mathsf{Bw}\rangle$, we
may write \[
\alpha \ =\ A\circ \widetilde\sigma,
\] where $A\in \mathrm{Aff}((\F^U)^d)$ is affine and $\sigma$ is an
automorphism of the ordered-field reduct of $\F_{\exp}^U$.

Because any ordered-field automorphism fixes $0$ and $1$, the
componentwise map $\widetilde\sigma$ fixes the named frame points:
$\widetilde\sigma(O)=O$ and $\widetilde\sigma(E_i)=E_i$ for each
$i=1,\dots,d$. Since $\alpha$ preserves each of $O,E_1,\dots,E_d$,
it follows that the affine map \[
A \ =\ \alpha\circ (\widetilde\sigma)^{-1}
\] also fixes $O,E_1,\dots,E_d$. Hence $A=\mathrm{id}$ (an affine
map fixing the origin and all standard basis points is the identity).
Therefore $\alpha=\widetilde\sigma$.

Since $\alpha$ preserves the primitive relation $\Gamma_{\exp}$, we
obtain commutativity with $\exp$. Indeed, for each $x\in F^U$ let \[
p=(x,0,\dots,0)\quad\text{and}\quad q=(0,\exp(x),0,\dots,0).
\] Then $\Gamma_{\exp}(p,q)$ holds. Applying
$\alpha=\widetilde\sigma$ gives \[
\Gamma_{\exp}\big((\sigma(x),0,\dots,0),\ (0,\sigma(\exp(x)),0,\dots,0)\big),
\] which by definition of $\Gamma_{\exp}$ implies
$\sigma(\exp(x))=\exp(\sigma(x))$ for all $x\in F^U$. Thus
$\sigma\in \mathrm{Aut}(\F_{\exp}^U)$.

Finally, since
$A=\mathrm{id}\in \mathrm{AffAut}(\G_{\exp}(\F_{\exp}^U))$, the
factorisation $\alpha=A\circ \widetilde\sigma$ witnesses semilinear
faithfulness.
\end{proof}

\begin{remark}
	\label{rem:8.7}
Here the affine-invariance condition is vacuous:
every relation is preserved by the trivial group
$\mathrm{AffAut}(\G_{\exp})$. The conclusion of Theorem \ref{thm:6.1} is
nevertheless meaningful: semilinear faithfulness implies that every
parameter-free $\F_{\exp}$-definable relation on $(F^d)^n$
is already a concept of $\G_{\exp}$. The named frame and the graph
relation recover enough of the expanded base structure to make that
conclusion possible.
\end{remark}

\section{Conclusion and further directions}
\label{sec:9}

We have isolated the affine-geometric input that drives the
representation theorem of \cite{MSS26}: the ultrapower argument
needs a semilinear factorisation of ultrapower automorphisms together
with a guarantee that the residual (coordinatewise) component preserves
the ambient definable relations under consideration. In the original FFD
setting, pure-field/pure-ordered-field definability provides that
guarantee automatically; in expanded settings we capture it explicitly
by \emph{semilinear faithfulness}. Under this hypothesis we recover an
Erlangen-style characterisation of the concepts of a geometry in terms
of invariance under its affine automorphism group (Theorem \ref{thm:6.1}) and the
corresponding dual inclusion theorem (Theorem \ref{thm:7.2}), and we have given
examples in which the primitives are not pure-field definable but the
same mechanism still applies.

A natural direction for further investigation is to remove the
specifically affine and coordinatised context and ask for a similar
``Erlangen via ultrapowers'' principle for more general base structures.

\subsection{A prospective generalisation beyond fields}
\label{sec:9.1}

Fix a one-sorted first-order structure $\M$ (not necessarily a field)
with universe $M$ and fix an integer $d\ge 1$. Let $Q_1,\dots,Q_k$
be parameter-free $\M$-definable relations on $M^d$ (of
arbitrary finite arities), and write \[
\P\ :=\ \langle M^d;\ Q_1,\dots,Q_k\rangle .
\] (We use the letter $\P$ for ``power structure'' to emphasise that we
are no longer in an intrinsically geometric setting.) By Remark \ref{rem:2.2}, we
may speak of the concepts $\mathrm{Conc}(\P)$ in exactly the same way
as for coordinate geometries.

For each ultrafilter $U$, define the induced structure \[
\P(\M^U)\ :=\ \langle (M^U)^d;\ Q_1^U,\dots,Q_k^U\rangle,
\] where each $Q_i^U$ is interpreted in $\M^U$ by the same
parameter-free $\M$-formula that defines $Q_i$ in
$\M$. 

To formulate an ``Erlangen via ultrapowers'' principle in this
generality one needs, in place of ``affine automorphisms'', a designated
class of ``geometric symmetries'' $H\le \mathrm{Aut}(\P)$ with two
properties:

\begin{enumerate}
\item
  invariance under $H$ is first-order expressible (so it transfers to
  ultrapowers), and
\item
  every ultrapower automorphism factors through $H(\M^U)$ and a
  coordinatewise lift of a base-structure automorphism.
\end{enumerate}

A convenient way to make the transfer step precise is to assume that
elements of $H$ are \emph{coded} inside $\M$.

\begin{assumption}[uniformly coded, $\M$-definable symmetry group]
	\label{ass:9.1}
Let $D \subseteq M^m$ be a parameter-free definable set, and let $D(h)$ be a formula defining
membership in this set. Suppose also that 
$\theta(x,y,h)$ is a parameter-free formula with $x,y\in M^d$ and $h\in M^m$. 
We assume that $\M$ satisfies the following first-order
requirements, and refer to the members of $D$ as \emph{admissible codes}:

\begin{enumerate}
\item
  For each admissible code $h$, the relation
  $\theta(\cdot,\cdot,h)$ is the graph of a total bijection
  $f_h:M^d\to M^d$. We call $f_h$ the map coded by $h$.
\item
  Every coded map $f_h$ respects every primitive relation $Q_i$ of $\P$.
\item
  The coded maps contain an identity and are closed under composition
  and inverse: there is a code for the identity map; for any two
  admissible codes there is an admissible code whose $\theta$-graph is
  the composite of their graphs; and for every admissible code there is
  one whose graph is the inverse graph.
\end{enumerate}
\end{assumption}

All three requirements can be written as a finite collection of
parameter-free first-order sentences in the language of $\M$.
For example, preservation of a primitive $Q_i$ is expressed by \[
\forall h\,\big(D(h)\rightarrow
\forall \bar x\,\forall \bar y\,
\big(\bigwedge_j\theta(x_j,y_j,h)\rightarrow
(Q_i(\bar x)\leftrightarrow Q_i(\bar y))\big)\Big),
\] where the displayed occurrence of $Q_i$ is replaced by its fixed
$\M$-definition.

For any structure $\mathcal N\equiv\M$, define \[
H(\mathcal N):=
\{\,f_h:\mathcal N\models D(h)\,\},
\] where $f_h$ is the map defined by the $\theta$-graph in
$\mathcal N$. In particular write $H:=H(\M)$.

\begin{lemma}[The coded symmetry group transfers]
	\label{lem:9.2}
Under
Assumption \ref{ass:9.1}, for every $\mathcal N\equiv\M$ one has \[
H(\mathcal N)\le \mathrm{Aut}(\P(\mathcal N)),
\] where $P(\mathcal N)$ is obtained by interpreting the formulas
defining $Q_1,\dots,Q_k$ in $\mathcal N$. In particular, \[
H(\M^U)\le \mathrm{Aut}(\P(\M^U))
\] for every ultrafilter $U$.
\end{lemma}

\begin{proof}
The requirements that admissible codes define total
bijections, respect all the relations $Q_i$, and are closed under
identity, composition and inverse were stated as first-order sentences.
They hold in $\M$ by Assumption \ref{ass:9.1} and therefore in every
elementarily equivalent structure $\mathcal N$. The first two
requirements place every coded map in $\mathrm{Aut}(\P(\mathcal N))$,
and the third makes the resulting set of maps a subgroup. Every
ultrapower $\M^U$ is elementarily equivalent to $\M$
by Łoś's theorem.
\end{proof}

Under Assumption \ref{ass:9.1}, the transfer step has a direct analogue of Lemma
\ref{lem:4.3}.

\begin{lemma}[Transfer of $H$-invariance to ultrapowers]
	\label{lem:9.3}
Assume
Assumption \ref{ass:9.1} holds. Let $R\subseteq (M^d)^n$ be parameter-free
$\M$-definable, and assume $R$ is preserved by
$H=H(\M)$. Then for every ultrafilter $U$, the induced
relation $R^U\subseteq ((M^U)^d)^n$ is preserved by
$H(\M^U)$.
\end{lemma}

\begin{proof}
Choose a parameter-free $\M$-formula
$\varphi(\bar x)$ defining $R$. The hypothesis that $H$ preserves
$R$ can be expressed in $\M$ by the sentence \[
\forall h\,\big(D(h)\rightarrow
\forall \bar x\,\forall \bar y\,
\big(\bigwedge_{i=1}^n \theta(x_i,y_i,h)\rightarrow
(\varphi(\bar x)\leftrightarrow \varphi(\bar y))\Big)\Big).
\] By Łoś's theorem the same sentence holds in $\M^U$, which
says exactly that every element of $H(\M^U)$ preserves
$R^U$.
\end{proof}

To replace semilinear faithfulness, we can now ask for an ultrapower
factorisation using the uniformly interpreted group $H(\M^U)$
rather than affine maps.

\begin{definition}[$H$-factorisation faithfulness]
\label{def:9.4}
Assume Assumption \ref{ass:9.1} holds. Say that $\P$ is \emph{$H$-factorisation faithful over $\M$} if
for every ultrafilter $U$ and every \[
\alpha\in \mathrm{Aut}(\P(\M^U)),
\] there exist $h\in H(\M^U)$ and
$\sigma\in \mathrm{Aut}(\M^U)$ such that \[
\alpha=h\circ \widetilde\sigma,
\] where $\widetilde\sigma:(M^U)^d\to(M^U)^d$ acts componentwise.
\end{definition}

With Lemma \ref{lem:9.3} in hand, the proof of Theorem \ref{thm:6.1} adapts verbatim, now
with Theorem \ref{thm:3.2} applied directly to $\P$ rather than through the
coordinate-geometry corollary.

\begin{proposition}[Erlangen characterisation under coded $H$-factorisation faithfulness]
	\label{prop:9.5}
Assume Assumption \ref{ass:9.1} holds and assume that $\P$ is $H$-factorisation faithful over $\M$ in the sense of Definition \ref{def:9.4}. Let $R\subseteq (M^d)^n$ be any
parameter-free $\M$-definable relation. Then \[
R\in \mathrm{Conc}(\P)\ \Longleftrightarrow\ \big(\text{$R$ is preserved by }H\big).
\]
\end{proposition}

\begin{proof}
The forward direction is immediate since concepts are
preserved by all automorphisms, hence by $H\le \mathrm{Aut}(\P)$.

For the converse, assume $R$ is preserved by $H$. By Lemma \ref{lem:9.3}, for
each ultrafilter $U$ the relation $R^U$ is preserved by
$H(\M^U)$. Fix such a $U$ and let
$\alpha\in \mathrm{Aut}(\P(\M^U))$ be arbitrary. By $H$-factorisation faithfulness, write $\alpha=h\circ \widetilde\sigma$ with
$h\in H(\M^U)$ and $\sigma\in \mathrm{Aut}(\M^U)$.
Then $h$ preserves $R^U$ by the transferred hypothesis, and
$\widetilde\sigma$ preserves $R^U$ because $R^U$ is parameter-free
definable in $\M^U$ and $\sigma$ is an automorphism of
$\M^U$. Hence $\alpha(R^U)=R^U$ for every ultrafilter $U$
and every $\alpha\in\mathrm{Aut}(\P(\M^U))$.

Applying Theorem \ref{thm:3.2} to the structure $\P$, we conclude that $R$
is parameter-free definable in $\P$, i.e.~$R\in \mathrm{Conc}(\P)$.
\end{proof}

These schematic results indicate what must be supplied in order to push
the present method beyond affine coordinate geometries: a robust source
of ultrapower splittings and a natural, uniformly definable coding of
the intended symmetry group $H$.

\bibliographystyle{alpha}
\bibliography{lrb2026}

@book{Coste2002,
author = {Coste, M.},
title = {An Introduction to Semialgebraic Geometry},
institution = {Institut de Recherche Math\'ematique de Rennes}, 
year = 2002,
url = {https://perso.univ-rennes1.fr/michel.coste/polyens/SAG.pdf}
}

@book{Klein1872,
author = {F.~Klein},
title = {Vergleichende Betrachtungen \"{u}ber neuere geometrische Forschungen},
publisher = {Andreas Deichert, Erlangen}, year = 1872,
utl = {https://www.gutenberg.org/files/38033/38033-h/38033-h.htm}
}

@article{MSS26,
author = {Madar\'asz, J. and Stannett, M. and Sz\'ekely, G.},
title = {Definable coordinate geometries over fields},
journal = {The Review of Symbolic Logic}, 
note = {published online}, year = 2026, 
pages = {1--32},
doi = {10.1017/S175502032610118X}
}

\end{document}